\documentclass[12pt]{amsart}
\usepackage{amssymb}
\usepackage[all]{xy}

\input xy
\xyoption{all}

\newtheorem{lemma}{Lemma}[section]
\newtheorem{corollary}[lemma]{Corollary}
\newtheorem{theorem}[lemma]{Theorem}
\newtheorem{proposition}[lemma]{Proposition}

\theoremstyle{definition}
\newtheorem{remark}[lemma]{Remark}
\newtheorem{definition}[lemma]{Definition}

\newtheorem{example}[lemma]{Example}

\begin{document}

\title{Singular equivalence of finite dimensional algebras with radical square zero}\footnotetext{This research was in part supported by a grant from IPM (No. 93170419).}

\author{A. R. Nasr-Isfahani}
\address{Department of Mathematics, University of Isfahan, P.O. Box: 81746-73441, Isfahan, Iran\\
  and School of Mathematics, Institute for Research in Fundamental Sciences (IPM), P.O. Box: 19395-5746, Tehran, Iran}
\email{nasr$_{-}$a@sci.ui.ac.ir / nasr@ipm.ir}



\begin{abstract} We prove that the Crisp and Gow's quiver operation on a finite quiver $Q$ produces a new quiver $Q'$ with fewer vertices, such that the finite dimensional algebras $kQ/J^2$ and $kQ'/J^2$ are singularly equivalent. This operation is a general quiver operation which includes as specific examples some operations which arise naturally in symbolic dynamics (e.g., (elementary) strong shift equivalent, (in-out) splitting, source elimination, etc.).
\end{abstract}

\maketitle


\section{Introduction}

Throughout this paper $Q$ will denote a quiver, and $k$ will denote an arbitrary
field. The path algebra of $Q$ with coefficients in $k$ denote by $kQ$, and its two-sided ideal generated by arrows denote by $J$. Let $kQ/J^2$ be the factor algebra modulo the ideal $J^2$. When $Q$ is a finite quiver, $kQ/J^2$ is a finite-dimensional algebra with radical square zero.

The singularity category of a finite dimensional algebra $A$, denoted $\mathbf{D}_{sg}(A)$, is the Verdier quotient category of the bonded derived category of finitely generated $A$-modules with respect to the full subcategory consisting of perfect complexes \cite{BU, O}. The singularity category $\mathbf{D}_{sg}(A)$ of $A$ is trivial if and only if $A$ has finite global dimension. Thus the singularity category of $A$ measures the homological singularity of $A$. Note that the singularity category $\mathbf{D}_{sg}(A)$ of $A$ is triangle equivalent to the full subcategory consisting of compact objects of stable derived category $\mathbf{K}_{ac}(A$-$Inj)$ of $A$ \cite{K}.

The Leavitt path algebra of $Q$ with coefficients in $k$, denoted
$L_k(Q)$, was introduced by P. Ara, M. A. Moreno
and E. Pardo in \cite{AMP} and by G. Abrams and G. Aranda Pino in
\cite{AA1}. Leavitt path algebras are generalization of the algebras
constructed by Leavitt in \cite{Leavitt} to produce rings without the Invariant Basis Number property (i.e., $R_R^m\cong
R_R^n$ as left $R$-module with $m\neq n$). These algebras can be
regarded as the algebraic counterpart of the analytic graph
$C^*$-algebras. Note that the Leavitt path algebra construction arises naturally in a variety of
different contexts as diverse as analysis, noncommutative algebra,
noncommutative algebraic geometry, symbolic dynamic, number theory
and representation theory.

Recently, X. W. Chen and D. Yang in \cite{CY} studied the derived category of differential graded $L_k(Q)^{op}$-modules and the homotopy category of acyclic complexes of injective $kQ/J^2$-modules. X. W. Chen and D. Yang proved that: "Let $Q$ be a finite quiver without sinks, then  there are triangle equivalences $\mathbf{K}_{ac}(kQ/J^2-Inj)\simeq \mathbf{D}(L_k(Q)^{op})$ and $\mathbf{D}_{sg}(kQ/J^2)\simeq \mathbf{perf}(L_k(Q)^{op})$ (which obtains from the first equivalence restricting to the compact objects)". They also showed that: "Let $Q$ and $Q'$ be finite quivers without sinks. Then the algebras $kQ/J^2$ and $kQ'/J^2$ are singularly equivalent if and only if there is a triangle equivalence $\mathbf{K}_{ac}(kQ/J^2$-$Inj)\simeq\mathbf{K}_{ac}(kQ'/J^2$-$Inj)$ if and only if there is a triangle equivalence $\mathbf{K}_{ac}(kQ/J^2$-$Proj)\simeq\mathbf{K}_{ac}(kQ'/J^2$-$Proj)$ if and only if the Leavitt path algebras $L_k(Q)$ and $L_k(Q')$ are derived equivalent if and only if the Leavitt path algebras $L_k(Q)$ and $L_k(Q')$ are graded Morita equivalent".

The aim of this paper is to show that the Crisp and Gow's operation on a finite quiver $Q$ produces a new quiver $Q'$, such that the finite dimensional algebras $kQ/J^2$ and $kQ'/J^2$ are singularly equivalent. In the proof we use the X. W. Chen and D. Yang's Theorem \cite{CY} and show that the Leavitt path algebras $L_k(Q)$ and $L_k(Q')$ are graded Morita equivalent. Several authors in \cite{AALP}, \cite{AA3}, \cite{ALPS} and \cite{H1} analyzed certain operations on quivers which preserve the (graded) Morita equivalence class of the associated Leavitt path algebras. This results are algebraic analogue of the analytic results that proved in \cite{A}, \cite{B}, \cite{BP}, \cite{CG}, \cite{D} and \cite{DT}. Note that in the above papers the authors explored the connections between symbolic dynamics and Leavitt path algebras (graph $C^*$-algebras).

The graded Morita equivalence property of Leavitt path algebras not only important in the contexts of Leavitt path algebras and singularity of path algebras, but also has applications in the noncommutative algebraic geometry. S. P. Smith in \cite{PS1} and R. Hazrat in \cite{H1} studied the connection between Leavitt path algebras and noncommutative algebraic geometry. Combining the results of Smith and Hazrat we can see that for finite quivers $Q$ and $Q'$ without sinks, if the Leavitt path algebras $L_k(Q)$ and $L_k(Q')$ are graded Morita equivalent then the categories $QGr(kQ)$ (resp. $qgr(kQ)$) and $QGr(kQ')$ (resp. $qgr(kQ')$) are equivalent, where $QGr(kQ)$ (resp. $qgr(kQ)$) is the category of quasi-coherent (resp. coherent) sheaves on noncommutative projective scheme $Proj_{nc}(kQ)$.

The paper is organized as follows. In section 2 we recall some
definitions and basic properties of Leavitt path algebras, differential graded algebras and noncommutative algebraic geometry that we need in the paper. In
section 3 we give our main result, Theorem \ref{theresult}, which
provide sufficient conditions for quivers $Q$ and $Q'$ such that the Leavitt path algebras $L_k(Q)$ and $L_k(Q')$
are graded Morita equivalent. Also in this section we give some corollaries and applications of the main Theorem. Note that the results of this section not only clarify the already known results, but also give a positive answer of the Hazrat's question in \cite[Remark 31.]{H1}, provide new examples of interesting quiver moves which preserve the graded Morita equivalence class of the associated Leavitt path algebras, and extend some known results. Finally, in the last section we give the proof of the main Theorem.


\section{Definitions and Preliminaries}

\subsection{Leavitt path algebras}
A {\it quiver} $Q=(Q^{0},Q^{1},r,s)$ consists of two
sets $Q^{0}$ and $ Q^{1}$ together with maps $r,s:Q^{1}\rightarrow
Q^{0}$. The elements of $ Q^{0}$ are called \textit{vertices} and
the elements of $Q^{1}$ are called \textit{edges}. A quiver $Q$ is said to be finite if $Q^0$ and $Q^1$ are finite sets. A vertex which emits no
edges is called \textit{\ sink}; a vertex which emits
infinitely-many edges is called \textit{infinite emitter}.

A finite path $\mu $ in a
quiver $Q$ is a finite sequence of edges $\mu =\alpha_{1}\dots \alpha_{n}$ such that $
r(\alpha_{i})=s(\alpha_{i+1})$ for $i=1,\dots ,n-1$. In this case, $n=l(\mu)$ is called the length
of $\mu $. Let $v\in Q^0$, we define $s(v) = r(v) = v$, and $l(v)=0$. For any $n\in {\mathbb N}$ the set of paths of length $n$ is denoted by $Q^n$.
We denote the set of all finite paths in $Q$ by $\text{Path}(Q)$ i.e., $\text{Path}(Q)=\bigcup_{n\in {\mathbb N}} Q^n$. Also we denote
by $\mu ^{0}$ the set of the vertices of the path $\mu=\alpha_{1}\dots \alpha_{n}$, that is, the set $
\{s(\alpha_{1}),r(\alpha_{1}),\dots ,r(\alpha_{n})\}$. A path $\mu=\alpha_{1}\dots \alpha_{n}$ is \textit{closed} if $r(\alpha_{n})=s(\alpha_{1})$,
in this case $\mu $ is said to be {\it based at the vertex} $s(\alpha_{1})$. If $\mu=\alpha_{1}\dots \alpha_{n}$ is a closed path based at $s(\alpha_{1})$ and $s(\alpha_{i})\neq s(\alpha_{j})$ for every $i\neq j$, then $\mu $ is called a \textit{cycle}.

 A sequence $\{\alpha_i\}_{i\in {\mathbb N}}$ of edges in $Q$ is called an infinite path in case $r(\alpha_i) = s(\alpha_{i+1})$ for all $i$ and we denote the set of all infinite paths in $Q$ by $Q^\infty$.
For each $\alpha\in Q^{1}$, we call $\alpha^{\ast }$ a {\it ghost edge}. We let $r(\alpha^{\ast}) $ denote $s(\alpha)$, and we let $s(\alpha^{\ast })$ denote $r(\alpha)$.

\begin{definition} {\rm Let $Q$ be an arbitrary quiver and $k$ an arbitrary field. The \textit{Leavitt path algebra} $L_{k}(Q)$ is defined to be the $k$-algebra generated by a set $\{v:v\in Q^{0}\}$ of pairwise orthogonal idempotents together with a set of variables $\{\alpha,\alpha^{\ast }:\alpha\in Q^{1}\}$ which satisfy the following
conditions:

(1) $s(\alpha)\alpha=\alpha=\alpha r(\alpha)$ for each $\alpha\in Q^{1}$.

(2) $r(\alpha)\alpha^{\ast }=\alpha^{\ast }=\alpha^{\ast }s(\alpha)$\ for each $\alpha\in Q^{1}$.

(3) (The ``CK-1 relations") For each $\alpha,\beta\in Q^{1}$, $\alpha^{\ast}\alpha=r(\alpha)$ and $
\alpha^{\ast}\beta=0$ if $\alpha\neq \beta$.

(4) (The ``CK-2 relations") For each $v\in Q^{0}$ with $0 < |s^{-1}(v)| <\infty$,
\begin{equation*}
v=\sum_{\{\gamma\in Q^{1},\ s(\gamma)=v\}}\gamma\gamma^{\ast}.
\end{equation*}}
\end{definition}

The relations (CK1) and (CK2) are called the "Cuntz-Krieger
relations" after the relations introduced by Cuntz and Krieger for
their eponymous $C^*$-algebras. Another definition for $L_k(Q)$ can be given using the extended quiver $\widehat{Q}$. This quiver has the same set of vertices $Q^0$ and also has the same set of edges $Q^1$ together with the so-called ghost edges $\alpha^*$ for each $\alpha\in Q^1$, whose directions are opposite those of the corresponding $\alpha\in Q^1$. Thus, $L_k(Q)$ can be defined as the usual path algebra $k\widehat{Q}$ subject to the Cuntz-Krieger relations.
The set of elements of a $k$-algebra $A$ satisfying these same relations is called a $Q$-family. Let $A$ be a $k$-algebra with a $Q$-family, then by the
Universal Homomorphism Property of $L_k(Q)$, there is a unique $k$-algebra homomorphism from $L_k(Q)$ to $A$ mapping the
generators of $L_k(Q)$ to their appropriate counterparts in $A$. We will refer to this property as the
Universal Homomorphism Property of $L_k(Q)$.

\smallskip
Recall that the classical Leavitt algebra, $L_k(1,n)$, is the free associative $k$-algebra withe generators $x_i, y_i, 1\leq i\leq n$ and relations $x_iy_j=\delta_{ij}$, for each $i, j$ and $\sum_{i=1}^{n}y_ix_i=1$. Many well-known algebras can be realized as the Leavitt path algebra
of a quiver. The classical Leavitt algebras $L_k(1,n)$ for $n\geq 2$ can be obtained by the rose with $n$ petals. The ring of Laurent polynomials $k[x,x^{-1}]$ is the Leavitt path algebra of the quiver given by a single loop. Matrix algebras ${\mathbb M}_n(k)$ can be realized by the $A_n$ (i.e. line quiver with $n$ vertices and $n-1$ edges)  (for more details see \cite{AA1}, \cite{AA2}, and \cite{S}).

A singular vertex is a vertex that is either a sink or an infinite
emitter, and we denote the set of singular vertices by
$Q^0_{\text{sing}}$. We refer to the elements of $Q^0_{\text{reg}}:= Q^0 \setminus
Q^0_{\text{sing}}$ as regular vertices. Note that a vertex $v\in Q^0$ is a regular vertex
if and only if $0 < |s^{-1}(v)| <\infty$.

We define a relation $\geq $ on $Q^{0}$ by setting $v\geq w$ if there exists
a path $\mu$ in $Q$ from $v$ to $w$, that is, $v=s(\mu)$ and $w=r(\mu)$. A subset $H$ of $Q^{0}$ is called \textit{hereditary} if for each $v\in H$, $v\geq w$ implies that $w\in H$. A subset $H \subseteq Q^0$ is called \textit{saturated} if for any regular vertex $v$, $r(s^{-1}(v))\subseteq H$ implies that $v\in H$. The set
$T(v) = \{w\in E^0 | v \geq w\}$ is called the tree of $v$, and it is the smallest hereditary subset of $Q^0$
containing $v$. The hereditary saturated closure of a set $X$ is defined as the smallest hereditary and saturated subset of $Q^0$ containing $X$. It is shown in \cite{AMP} that the hereditary saturated closure of a set $X$ is $\overline{X} =\bigcup_{n=0}^{\infty}\Lambda_n(X)$, where
\begin{enumerate}
\item[{\rm (a)}] $\Lambda_0(X)=T(X)$,
\item[{\rm (b)}] $\Lambda_n(X)=\{y\in Q^0 | s^{-1}(y)\neq\emptyset, r(s^{-1}(y))\subseteq \Lambda_{n-1}(X)\}\cup \Lambda_{n-1}(X)$, for $n\geq 1$.
\end{enumerate}

\subsection{Differential graded algebras}
The definitions of this section are borrowed from \cite[section 2]{CY} and \cite{H2}.

Let $A$ be a $G$-graded algebra. A right $A$-module $M$ is called a graded right $A$-module, if there exists a
direct sum decomposition $M=\oplus_{g\in G}M_g$, where each $M_g$ is an additive subgroup of $M$ such that
$M_gA_h\subseteq M_{g+h}$ for all $g, h\in G$. Elements $m\in M_g$ are said to be homogeneous of degree $g$, and are denoted by $|m|=g$. We denote by $Gr$-$A$ the category of graded right $A$-modules with degree preserving morphisms, or $Gr^G$-$A$ to emphasis the category is $G$-graded.

One of the useful properties of $L_k(Q)$ is that it is a graded algebra. Let $G$ be a group with the identity element $e$ and $w:Q^1\rightarrow G$ be a weight map. Also let $w(\alpha^\ast)=w(\alpha)^{-1}$ and $w(v)=e$, for each $\alpha\in Q^1$ and $v\in Q^0$. Then the path algebra $k\widehat{Q}$ of the extended quiver $\widehat{Q}$ is a $G$-graded $k$-algebra and since Cuntz-Krieger relations are homogeneous, $L_k(Q)$ is a $G$-graded $k$-algebra. The natural grading given to a Leavitt path algebra is a $\mathbb{Z}$-grading by setting $w(\alpha)=1$,  $w(\alpha^\ast)=-1$ and $w(v)=0$, for each $\alpha\in Q^1$ and $v\in Q^0$. In this case the Leavitt path algebra can be decomposed as a direct sum of homogeneous components $L_k(Q)=\bigoplus_{n\in {\mathbb Z}} L_k(Q)_n$ satisfying $L_k(Q)_nL_k(Q)_m\subseteq L_k(Q)_{n+m}$. Actually, $$L_k(Q)_n=\text{span}_k\{pq^*: p,q\in \text{Path}(Q) , l(p)-l(q)=n \}.$$

Every element $x\in L_k(Q)_n$ is a homogeneous element of degree $n$.
An ideal $I$ is graded if it inherits the grading of $L_k(Q)$, that is, if $I=\bigoplus_{n\in {\mathbb Z}} (I\cap L_k(Q)_n)$.
Tomforde in \cite{T} (see also \cite[Theorem 3.5]{AMM}) proved that: (Graded Uniqueness Theorem) "Let $Q$ be a quiver and let $L_k(Q)$
be the associated Leavitt path algebra with the usual $\mathbb Z$-grading. If $A$ is a $\mathbb Z$-graded ring, and
$\pi :L_k(Q)\rightarrow A$ is a graded ring homomorphism with $\pi(v)\neq 0$ for all $v\in Q^0$, then $\pi$ is injective."

\begin{definition} (\cite[Definition 2.3.2]{H2}, \cite[Definition 1.]{H1}) {\rm Let $A$ and $B$ be $G$-graded algebras.

(1) Let $M$ be a graded right
$A$-module and $g\in G$. The graded right $A$-module $M(g) =\oplus_{h\in G}M(g)_h$, where $M(g)_h=M_{g+h}$, is called the \textit{$g$-suspended}, or \textit{$g$-shifted graded right $A$-module}.

(2) For each $g\in G$, the \textit{$g$-suspension functor} $\mathcal{T}_g:Gr$-$A\rightarrow Gr$-$A$, $M\mapsto M(g)$ is an isomorphism with the property $\mathcal{T}_g\mathcal{T}_h=\mathcal{T}_{g+h}$, for each $g, h\in G$.

(3) A functor $\phi:Gr$-$A\rightarrow Gr$-$B$ is called a \textit{graded functor} if $\phi\mathcal{T}_g=\mathcal{T}_g\phi$ for each $g\in G$.

(4) A graded functor $\phi:Gr$-$A\rightarrow Gr$-$B$ is called a \textit{graded equivalence} if there is a graded functor $\psi:Gr$-$B\rightarrow Gr$-$A$ such that $\psi\phi\cong_{gr}1_{Gr\textsc{-}A}$ and $\phi\psi\cong_{gr}1_{Gr\textsc{-}B}$.

(5) If there is a graded equivalence between $Gr$-$A$ and $Gr$-$B$, we say $A$ and $B$ are \textit{graded equivalent} or \textit{graded Morita equivalent} and we write $Gr$-$A\approx_{gr}Gr$-$B$, or $Gr^G$-$A\approx_{gr}Gr^G$-$B$ to emphasis the categories are $G$-graded.

(6) A functor $\phi:Mod$-$A\rightarrow Mod$-$B$ is called \textit{graded functor} if there is a graded functor $\phi':Gr$-$A\rightarrow Gr$-$B$ such that the following diagram, where the vertical functors are forgetful functor, commutes.

\[ \xymatrix{
  Gr\textsc{-}A\ar[d]_F \ar[r]^{\phi'}
& Gr\textsc{-}B  \ar[d]^F
\\
  Mod\textsc{-}A \ar[r]_\phi & Mod\textsc{-}B }
\]

The functor $\phi'$ is called an \textit{associated graded functor} of $\phi$.

(7) A functor $\phi:Mod$-$A\rightarrow Mod$-$B$ is called a \textit{graded equivalence} if it is graded and an equivalence.
}
\end{definition}

A differential graded algebra (or dg algebra) is a graded algebra $A$ with
a differential $d:A\rightarrow A$ of degree one such that $d(ab) = d(a)b + (-1)^{|a|}ad(b)$ for
homogeneous elements $a, b \in A$. We take a Leavitt path algebra $L_k(Q)$ as a differential graded algebra with a trivial differential. A right differential graded $A$-module (or right dg $A$-module) $M$ is a right
graded $A$-module $M=\oplus_{n\in \mathbb Z}M_n$ with a differential $d_M:M\rightarrow M$ such that $d_M(ma) = d_M(m)a + (-1)^{|m|}md(a)$ for homogeneous elements $a\in A$
and $m\in M$. A morphism between dg $A$-modules is a morphism of graded $A$-modules which
commutes with the differentials.
Let $C(A)$ be the category of dg $A$-modules and $M$ be a graded $A$-module. Chen and Yang in \cite{CY} defined a graded $A$-module $M[1]$ such that $M[1]_n = M_{n+1}$
and defined the $A$-action $\circ$ on $M[1]$ by $m\circ a = (-1)^{|a|}ma$. This gives rise to
the translation functor $[1] : Gr\textsc{-}A\rightarrow Gr\textsc{-}A$ of the category $Gr\textsc{-}A$.
The translation functor $[1]: C(A)\rightarrow C(A)$ is defined similarly and the differential is given by $d_{M[1]} =
-d_M$. We denote by $K(A)$ the homotopy category of right dg $A$-modules. Note that $K(A)$ is a triangulated
category. We denote
by $D(A)$ the derived category of dg $A$-modules, which is obtained from $C(A)$ by
formally inverting all quasi-isomorphisms. Denote by $Mod\textsc{-}A$ the abelian category of right $A$-modules, by $Inj\textsc{-}A$ the category of injective right $A$-modules and by $Proj\textsc{-}A$ the category of projective right $A$-modules. Consider $K(Inj\textsc{-}A)$ (resp. $K(Proj\textsc{-}A)$) the homotopy category of complexes of injective (resp. projective) A-modules, which is a triangulated
subcategory of $K(A)$. A complex $X\in C(A)$ is called acyclic if $H^n(X)=0$, for all $n\in \mathbb{Z}$. The triangulated subcategory of $K(Inj\textsc{-}A)$ (resp. $K(Proj\textsc{-}A)$) consisting of acyclic complexes will be denoted by $K_{ac}(Inj\textsc{-}A)$ (resp. $K_{ac}(Proj\textsc{-}A)$).

We denote by $Q^{op}$, the opposite quiver of $Q$ and by $L_k(Q)^{op}$,
the opposite algebra of $L_k(Q)$, which is also viewed as a differential graded algebra
with trivial differential. Chen and Yang in \cite{CY} proved the following proposition. Since we need this proposition in this paper we record it here.

\begin{proposition}\label{pro} \cite[Proposition 6.4]{CY} Let $Q$ and $Q'$ be finite quivers without sinks. Then the following statements are equivalent:

(1) The algebras $kQ/J^2$ and $kQ'/J^2$ are singularly equivalent;

(2) The Leavitt path algebras $L_k(Q)$ and $L_k(Q')$ are graded Morita equivalent;

(3) The Leavitt path algebras $L_k(Q)$ and $L_k(Q')$ are derived equivalent;

(4) The opposite Leavitt path algebras $L_k(Q)^{op}$ and $L_k(Q')^{op}$ are derived equivalent;

(5) There is a triangle equivalence $K_{ac}(Inj\textsc{-}kQ/J^2)\xrightarrow[]{\sim} K_{ac}(Inj\textsc{-}kQ'/J^2)$;

(6) There is a triangle equivalence $K_{ac}(Proj\textsc{-}kQ^{op}/J^2)\xrightarrow[]{\sim}K_{ac}(Proj\textsc{-}kQ'^{op}/J^2)$.

\end{proposition}

\subsection{Noncommutative algebraic geometry}

Let $\Lambda$ be a noncommutative algebra, $Gr$-$\Lambda$ the category of $\mathbb{Z}$-graded right $\Lambda$-modules and $Fdim$-$\Lambda$ the full subcategory of $Gr$-$\Lambda$ consisting of modules that are the sum of their finite dimensional submodules. The category $Fdim$-$\Lambda$ is a Serre subcategory of $Gr$-$\Lambda$. The quotient category $QGr(\Lambda)=Gr\textsc{-}\Lambda/Fdim\textsc{-}\Lambda$ is one of the fundamental constructions in noncommutative algebraic geometry. About 60 years ago, Serre proved that: If $\Lambda$ is a commutative algebra which is generated by a finite number of
elements of degree one, then the category of quasi-coherent
sheaves on the scheme $Proj(\Lambda)$ is equivalent to $QGr(\Lambda)$. In Noncommutative algebraic geometry a noncommutative scheme $Proj_{nc}(\Lambda)$ is defined such that the category of quasi-coherent sheaves on $Proj_{nc}(\Lambda)$ is $QGr(\Lambda)$. When $\Lambda$ is a coherent algebra, $qgr(\Lambda):= gr\textsc{-}\Lambda/fdim\textsc{-}\Lambda$ is viewed
as the category of coherent sheaves on $Proj_{nc}(\Lambda)$, where $gr\textsc{-}\Lambda$ is the full subcategory of $Gr\textsc{-}\Lambda$ consisting of finitely presented graded right $\Lambda$-modules and $fdim\textsc{-}\Lambda$ is the full subcategory of $gr\textsc{-}\Lambda$ consisting of finite dimensional modules. Paul Smith in \cite{PS1} studied the connection between the category $QGr(kQ)$, where $Q$ is a finite quiver, and the category of graded modules over Leavitt path algebra $L_k(Q)$. He also proved that: "If $Q$ is a finite quiver, then the triangulated categories $\mathbf{D}_{sg}(kQ/J^2)$ and $qgr(kQ)$ are equivalent" \cite[Theorem 1.6]{PS1}. Let $Q$ and $Q'$ be finite quivers without sinks, if the Leavitt path algebras $L_k(Q)$ and $L_k(Q')$ are graded Morita equivalent, then by Proposition \ref{pro} and \cite[Theorem 1.6]{PS1} we have $qgr(kQ)$ and $qgr(kQ')$ are equivalent.

Let $A$ and $B$ be square matrices with entries in $\mathbb{N}$. If there is a pair of matrices $L$ and $R$ with non-negative integer
entries such that $A = LR$ and $B = RL$ then we say that $A$ and $B$ are elementary strong shift equivalence \cite[Definition 7.2.1]{LM}.
$A$ and $B$ are called strong shift equivalent if there is a chain of elementary
strong shift equivalences from $A$ to $B$. Also $A$ and $B$ are called shift equivalent if there are non-negative integer matrices $C$ and $D$ such that $A^n=CD$ and $B^n=DC$, for some $n\in \mathbb{N}$, $AC=CB$ and $DA=BD$. Note that the notion of (strong) shift equivalence comes from symbolic dynamics (for more details see \cite{LM}).  Recently Smith in \cite{PS} proved that: "If the incidence matrices of quivers $Q$ and $Q'$ are strong shift equivalent, then $QGr(kQ)$ and $QGr(kQ')$ are equivalent". Hazrat in \cite[Proposition 15]{H1} proved that: "For finite quivers $Q$ and $Q'$ with no sinks, if $L_k(Q)$ and $L_k(Q')$ are graded Morita equivalent, then the incidence matrices of quivers $Q$ and $Q'$ are shift equivalent". He also proved that: "For finite quivers $Q$ and $Q'$ with no sinks, if the incidence matrices of quivers $Q$ and $Q'$ are shift equivalent, then $QGr(kQ)$ and $QGr(kQ')$ are equivalent" (\cite[Corollary 24]{H1}). Combining this results we have: for finite quivers $Q$ and $Q'$ with no sinks if $L_k(Q)$ and $L_k(Q')$ are graded Morita equivalent, then $QGr(kQ)$ and $QGr(kQ')$ are equivalent.


\section{Main theorem and its applications}

In this section we give the main theorem of this paper and some
applications of the main theorem. Also we show that several known
results are corollaries of the main theorem.

The following
theorem is our main theorem, which we will prove it in the section 4. We
denote by $(1/n)\mathbb{Z}$, the cyclic subgroup of the rational numbers generated by $1/n$.

\begin{theorem}\label{theresult} Let $Q$ be a quiver with finite number of vertices (but it can have vertices which are infinite emitters), $k$ an arbitrary field, $Q'^0\subset Q^0$ such that $Q^0_{\text{sing}}\subseteq Q'^0$ and if $\mu\in \text{Path}(Q)$ is a cycle then $Q'^0\cap \mu ^{0}\neq \emptyset$. Let $Q'$ be the quiver with vertex set $Q'^0$ and one edge $e_\mu$ for
each path $\mu\in \text{Path}(Q)\backslash Q^0$ with $s(\mu),
r(\mu)\in Q'^0$ and $r(\mu_i)\in Q^0\backslash Q'^0$ for $1\leq
i<l(\mu)$, such that $s(e_\mu)=s(\mu)$ and $r(e_\mu)=r(\mu)$, and let
$n=max\{l(\mu)|\mu\in \text{Path}(Q)\backslash Q^0, s(\mu),
r(\mu)\in Q'^0\hskip.1cm and\hskip.1cm  r(\mu_i)\in Q^0\backslash Q'^0\hskip.1cm for\hskip.1cm 1\leq
i<l(\mu)\}$. Then $L_k(Q')$ and $L_k(Q)$ are
$(1/n)\mathbb{Z}$-graded Morita equivalent.
\end{theorem}

\begin{remark}  {\rm Note that for a given quiver $Q$ and
different subsets $Q'^0, Q''^0\subset Q^0$ which satisfy the hypotheses of the
Theorem \ref{theresult} we get different graded Morita
equivalences. For example, let
$$\begin{matrix}Q:\xymatrix{{\bullet_1}\ar[r]^\alpha&{\bullet_2}\ar[r]^\beta&{\bullet_3}}
\end{matrix}\hskip.5cm$$ $Q'^0=\{2, 3\}$ and $Q''^0=\{1, 3\}$. Then $$\begin{matrix}Q':\xymatrix{{\bullet_2}\ar[r]^\beta&{\bullet_3}}
\end{matrix}\hskip.5cm$$ $$\begin{matrix}Q'':\xymatrix{{\bullet_1}\ar[r]&{\bullet_3}}
\end{matrix}\hskip.5cm$$ By Theorem \ref{theresult}, $L_k(Q')$ and $L_k(Q)$ are
$\mathbb{Z}$-graded Morita equivalent, but $L_k(Q'')$ and $L_k(Q)$
are $(1/2)\mathbb{Z}$-graded Morita equivalent.}
\end{remark}

\begin{remark}  {\rm Note that in Theorem \ref{theresult}, $L_k(Q')$ and $L_k(Q)$ are not $\mathbb{Z}$-graded Morita equivalent in general. For example, let $$\begin{matrix}Q:\xymatrix{{\bullet_1}\ar@/^/[rr]^\alpha&&{\bullet_2}\ar@/^/[ll]^{\beta}}
\end{matrix}\hskip.5cm$$ and $Q'^0=\{1\}$. Then $L_k(Q')$ and $L_k(Q)$ are not $\mathbb{Z}$-graded Morita equivalent but they are $(1/2)\mathbb{Z}$-graded Morita equivalent.
}
\end{remark}

Let $G$ be a group and $A=\bigoplus_{g\in G}A_g$ be a $G$-graded algebra. The set $G_A=\{g\in G| A_g\neq 0\}$ is called the support of $A$.

\begin{remark}  {\rm Note that in Theorem \ref{theresult} the support of $L_k(Q')$, $(1/n)\mathbb{Z}_{L_k(Q')}$ is $\mathbb{Z}$, but the support of $L_k(Q)$, $(1/n)\mathbb{Z}_{L_k(Q)}$, is not equal to $\mathbb{Z}$ in general.}
\end{remark}

\begin{example}\label{exa}(i) Consider the quiver $Q$ given by
$$
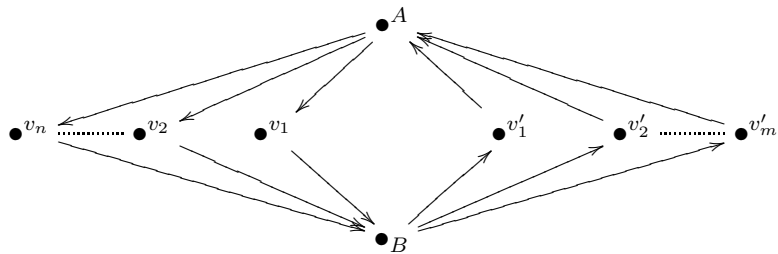
\xymatrix{
&&&&&&{\bullet}^A \ar[dl]\ar[dll] \ar@{>}[dlll] \ar@{<-}[drrr]\ar@{<-}[drr] \ar@{<-}[dr]\\
&&&{\bullet}^{v_n}\ar@{.}[r] \ar@{>}[drrr]&{\bullet}^{v_2}\ar[drr]&{\bullet}^{v_1}\ar[dr]&&{\bullet}^{v'_1}\ar@{<-}[dl]&{\bullet}^{v'_2}\ar@{.}[r]\ar@{<-}[dll]&{\bullet}^{v'_m}\ar@{<-}[dlll]&\\&&&&&&{\bullet}_B}
$$

Letting $Q'^0=\{A\}$, one can observe that the hypothesis of Theorem \ref{theresult} are satisfied. Then $L_k(Q)$ and $L_k(1,mn)$ are $(1/4)\mathbb{Z}$-graded Morita equivalent.\\

\label{exb}(ii) Consider the quiver $Q$ given by

$$
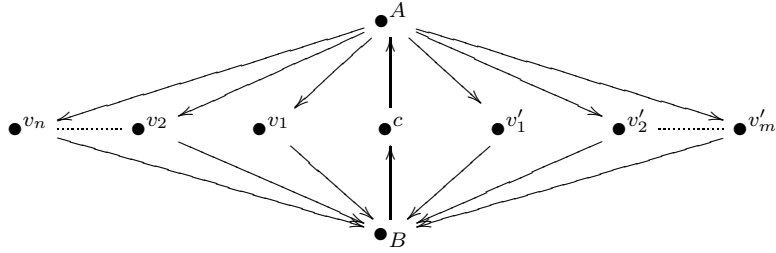
\xymatrix{
&&&&&&&{\bullet}^A \ar@{<-}[d] \ar[dl]\ar[dll] \ar@{->}[dlll] \ar@{->}[drrr] \ar[drr]  \ar[dr]\\
&&&&{\bullet}^{v_n}\ar@{.}[r]\ar@{->}[drrr] &{\bullet}^{v_2}\ar[drr]&{\bullet}^{v_1}\ar[dr]&{\bullet}^{c}\ar@{<-}[d]&{\bullet}^{v'_1}\ar[dl]&{\bullet}^{v'_2}\ar@{.}[r]\ar[dll]&{\bullet}^{v'_m}\ar@{->}[dlll]&\\&&&&&&&{\bullet}_B}$$

Let $Q'^0=\{A\}$, it is easy to see that the hypothesis of Theorem
\ref{theresult} are satisfied and hence $L_k(Q)$ is
$(1/4)\mathbb{Z}$-graded Morita equivalent to $L_k(1,m+n)$.
\end{example}

\subsection{Elementary strong shift equivalence}

\begin{definition}\label{def1}(\cite[Definition 5.1]{B})
Let $Q_i=(Q_i^0,Q_i^1,r_i,s_i)$ for $i=1,2$ be quivers. Suppose that
there is a quiver $Q_3=(Q_3^0,Q_3^1,r_3,s_3)$ which has the following properties:

(i) $Q_3^0=Q_1^0\cup Q_2^0$, and $Q_1^0\cap Q_2^0=\emptyset$.

(ii) $Q_3^1=Q_{12}^1\cup Q_{21}^1$ where $Q_{ij}^1:=\{e\in Q_3^1 : s_3(e)\in Q_i^0, r_3(e)\in Q_j^0\}$.

(iii) For $i\in \{1,2\}$, there exist source and range-preserving bijections $\theta_i: Q_i^1\rightarrow Q_3^2(Q_i^0,Q_i^0)$
where for $i\in \{1,2\}$, $Q_3^2(Q_i^0,Q_i^0):=\{\alpha\in Q_3^2 : s_3(\alpha)\in Q_i^0, r_3(\alpha)\in Q_i^0\}$.

Then we say that $Q_1$ and $Q_2$ are elementary strong shift equivalent $(Q_1\sim_{ES}Q_2)$ via $Q_3$.
We define strong shift equivalence (denoted by $\sim_S$) to be the equivalence relation on
row-finite quivers generated by elementary strong shift equivalence. It is easy to see that finite quivers $Q_1$ and $Q_2$ are elementary strong shift equivalent if and only if incidence matrices $B_{Q_1}$ and $B_{Q_2}$ are elementary strong shift equivalence. Note that in this case if  $B_{Q_1}=LR$ and $B_{Q_2}=RL$, where $L$ and $R$ are matrices with non-negative integer
entries, then $B_{Q_3}=\left(\begin{array}{cc}
0&L\\
R&0\\
\end{array}\right)$. \\\\
\end{definition}

\begin{example}\label{example 3}
Let
$\begin{matrix}Q_1:\xymatrix{{\bullet}\ar[r]&{\bullet}\ar@(ur,dr)}
\end{matrix}\hskip.5cm$ , $\hskip1cm$ $\begin{matrix}Q_2:\xymatrix{{\bullet}\ar@<.5ex>[r]\ar@<-.5ex>[r]&{\bullet}\ar@(ur,dr)}
\end{matrix}\hskip.5cm$ and

 $$\begin{matrix}Q_3: \xymatrix{
&{\bullet}\ar@<.5ex>[d]\ar@<-.5ex>[d]&{\bullet}\ar[d]\\
&{\bullet}\ar[ur]&{\bullet}\ar@<1ex>[u]}
\end{matrix}$$

Then it is easy to see that $Q_1$ and $Q_2$ are elementary strong shift equivalent via $Q_3$.\\
\end{example}

\begin{corollary}\label{cor2} Let $Q_1$ and $Q_2$ be finite quivers with no
sinks which are elementary strong shift equivalent via the quiver
$Q_3$. Then the Leavitt path algebras $L_k(Q_1)$ and $L_k(Q_2)$ are
$(1/2)\mathbb{Z}$-graded Morita equivalent. Also if
$Q_1\sim_{S}Q_2$, then $L_k(Q_1)$ and $L_k(Q_2)$ are
$(1/2)\mathbb{Z}$-graded Morita equivalent.
\end{corollary}

\begin{proof}
Setting $Q=Q_3$ and $Q'^0=Q_1^0$, then $max\{l(\mu)|\mu\in \text{Path}(Q)\backslash Q^0, s(\mu),
r(\mu)\in Q'^0\hskip.1cm and\hskip.1cm  r(\mu_i)\in Q^0\backslash Q'^0\hskip.1cm for\hskip.1cm 1\leq
i<l(\mu)\}=2$. Thus by Theorem \ref{theresult}, $L_k(Q_1)$ and $L_k(Q_3)$ are
$(1/2)\mathbb{Z}$-graded Morita equivalent. Now setting $Q=Q_3$ and $Q'^0=Q_2^0$, again by Theorem \ref{theresult}, $L_k(Q_2)$ and $L_k(Q_3)$ are
$(1/2)\mathbb{Z}$-graded Morita equivalent and the result follows.
\end{proof}

Hazrat in Theorem 30 of \cite{H1} proved Corollary
\ref{cor2} and asked the following question:
"Let $Q_1$ and $Q_2$ be finite quivers with no sinks which are
elementary strong shift equivalent via the quiver $Q_3$. Whether
$L_k(Q_1)$ and $L_k(Q_2)$ are actually $\mathbb{Z}$-graded Morita
equivalent?" \\
The following proposition gives the positive answer of the Hazrat's
question.

\begin{proposition}\label{l} Let $H$ be a subgroup of a group $G$, $R$ and $S$ be $G$-graded rings such that $G_R, G_S\subseteq H$. If $R$ and $S$ are graded Morita equivalent as $G$-graded rings, then $R$ and $S$ are graded Morita equivalent as $H$-graded rings.
\end{proposition}

\begin{proof} Assume that $R$ and $S$ are graded Morita equivalent as $G$-graded rings. By Theorem 2.3.7 of \cite{H2} (see also Proposition 5.3 and Theorem 5.4 of \cite{BG}), Mod-$R$ is $G$-graded equivalent to Mod-$S$. Thus there exists a functor $\Phi: Mod$-$R\longrightarrow Mod$-$S$, which is $G$-graded and equivalence. So there exists a $G$-graded functor $\Phi': Gr^G$-$R\longrightarrow Gr^G$-$S$ such that the following diagram, where the vertical functors are forgetful
functors, is commutative.
\[ \xymatrix{
  Gr^G\textsc{-}R\ar[d]_F \ar[r]^{\phi'}
& Gr^G\textsc{-}S  \ar[d]^F
\\
  Mod\textsc{-}R \ar[r]_\phi & Mod\textsc{-}S }
\]
Since $\Phi'$ is $G$-graded functor, for each $g\in G$,
$\Phi'\mathcal{T}_g=\mathcal{T}_g\Phi'$, where
$\mathcal{T}_g:Gr^G$-$R\rightarrow Gr^G$-$R$, $M\mapsto M(g)$ is
$g$-suspension functor. Choose a set $U$ of left coset
representatives of $H$ in $G$, and let $\chi:Gr^G$-$S\longrightarrow
Gr^H$-$S$ be defined by $\chi(M)_h=\bigoplus_{u\in
U}M_{uh}$, for each $h\in H$ (since $G_S\subseteq H$,
$\chi(M)\in Gr^H$-$S$). Let
$\phi''=\chi\phi'\psi:Gr^H$-$R\longrightarrow Gr^H$-$S$, where
$\psi:Gr^H$-$R\longrightarrow Gr^G$-$R$ such that for
$M=\bigoplus_{h\in H}M_h\in Gr^H$-$R$, $\psi(M)=\bigoplus_{g\in G}M_g$,
and $M_g=0$ for each $g\in G\setminus H$ (since $G_R\subseteq H$, $\psi(M)\in Gr^G$-$R$). We claim that $\phi''$ is an associated
$H$-graded functor of $\phi$. For each $h_1\in H$ and $M\in
Gr^H$-$R$,
$\phi''\mathcal{T}_{h_1}(M)=\chi\phi'\psi\mathcal{T}_{h_1}(M)$.
Since $\Phi'$ is a $G$-graded functor,
$\phi'\mathcal{T}_{h_1}=\mathcal{T}_{h_1}\phi'$. Thus
$\chi\phi'\psi\mathcal{T}_{h_1}(M)=\chi
\mathcal{T}_{h_1}\phi'\psi(M)=\chi(\bigoplus_{g\in
G}(\phi'\psi(M))_{gh_1})=\bigoplus_{u\in U, g\in
G}(\phi'\psi(M))_{ugh_1}$.
$\mathcal{T}_{h_1}\phi''(M)=\mathcal{T}_{h_1}(\bigoplus_{u\in U,
g\in G}(\phi'\psi(M))_{ug})=\bigoplus_{u\in U, g\in
G}(\phi'\psi(M))_{ugh_1}$. Then
$\phi''\mathcal{T}_{h_1}=\mathcal{T}_{h_1}\phi''$ and so $\phi''$ is
a graded functor. Also the following diagram is commutative.
  \[ \xymatrix{
  Gr^H\textsc{-}R\ar[d]_F \ar[r]^{\phi''}
& Gr^H\textsc{-}S  \ar[d]^F
\\
  Mod\textsc{-}R \ar[r]_\phi & Mod\textsc{-}S }
\]
Thus $\phi''$ is an associated $H$-graded functor of $\phi$ and
hence $R$ and $S$ are graded Morita equivalent as $H$-graded rings.
\end{proof}

\begin{corollary}\label{cor12} Let $Q_1$ and $Q_2$ be finite quivers with no
sinks which are elementary strong shift equivalent via the quiver
$Q_3$. Then the Leavitt path algebras $L_k(Q_1)$ and $L_k(Q_2)$ are
$\mathbb{Z}$-graded Morita equivalent. Also if $Q_1\sim_{S}Q_2$,
then $L_k(Q_1)$ and $L_k(Q_2)$ are $\mathbb{Z}$-graded Morita
equivalent.
\end{corollary}

\begin{proof} By Corollary \ref{cor2}, $L_k(Q_1)$ and $L_k(Q_2)$ are
$(1/2)\mathbb{Z}$-graded Morita equivalent. But the supports of $L_k(Q_1)$ and $L_k(Q_2)$ are equal to $\mathbb{Z}$ and so the result follows by Proposition \ref{l}.
\end{proof}

By using of the Corollary \ref{cor12} we have the following
corollary which shows that, the main result of \cite{PS},
Proposition 15.(2) and Corollary 24 of \cite{H1} are corollaries of
Theorem \ref{theresult}.

\begin{corollary}\label{cor3}
i) Let $L$ and $R$ be $\mathbb{N}$-valued matrices such that $LR$ and $RL$
make sense. Let $Q^{LR}$ be the quiver with incidence matrix $LR$ and $Q^{RL}$ the
quiver with incidence matrix $RL$. Then $L_k(Q^{LR})$ and $L_k(Q^{RL})$ are graded Morita equivalent.\\
ii) If the incidence matrices of $Q$ and $Q'$ are strong shift equivalent, then $L_k(Q)$ and $L_k(Q')$ are graded Morita equivalent.
\end{corollary}

A shift space $(X, \alpha)$ over a finite alphabet $A$ is a compact subset $X$ of $A^{\mathbb{Z}}$
that is stable under the shift map $\alpha$ defined by $\alpha(f)(n) = f(n + 1)$. Let $Q$ be a finite quiver with no sources and sinks. The shift space $X_Q$, whose alphabet is the set of arrows in $Q$ is called the edge shift of $Q$ \cite[Definition 2.2.5]{LM}. Edge shifts are subshifts of
finite type. Also one may associate a subshift of finite type $X_M$ to every square $0-1$
matrix $M$ having non-zero rows or columns \cite[Definition 2.3.7]{LM}. The alphabet
is the set of vertices for the quiver whose incidence matrix is
$M$. Shift spaces $(X, \alpha_X)$ and $(Y, \alpha_Y )$ are conjugate (equivalence), or topologically conjugate, if there is a homeomorphism $\varphi: X\rightarrow Y$ such that $\alpha_Y\varphi =\varphi\alpha_X$. Williams's Theorem explains the importance of strong shift
equivalence:\\
\begin{theorem}(Williams's Theorem \cite[Theorem A]{W}) Let $M$ and $N$ be square $\mathbb{N}$-valued
matrices and $X_M$ and $X_N$ the associated subshifts of finite type. Then $X_M\cong X_N$ if and only if $M$ and $N$ are strong shift equivalent.
\end{theorem}
By using Williams's
Theorem and Corollary \ref{cor3} we have the following corollary.

\begin{corollary}\label{cor4} Let $(X, \alpha)$ and $(X, \beta)$ be subshifts of finite type, $Q$
and $Q'$ be quivers such that $(X, \alpha)=X_Q$ and $(X,
\beta)=X_{Q'}$. If $(X, \alpha)$ and $(X, \beta)$ are conjugate,
then $L_k(Q)$ and $L_k(Q')$ are $\mathbb{Z}$-graded Morita
equivalent.
\end{corollary}

\subsection{Splitting}

The notion of (in, out)-splitting, introduced by D. Lind and B. Marcus in the context of symbolic dynamics \cite{LM}. This notion in the contexts of graph $C^*$-algebras and Leavitt path algebras was first investigated in \cite{BP} and \cite{ALPS}.

 \begin{definition}(\cite[Section 5]{BP}, \cite[Definition 1.9]{ALPS})  Let $Q=(Q^0, Q^1, r, s)$ be a quiver. For each $v \in Q^0$ with $r^{-1}(v)\neq \emptyset$, partition the set $r^{-1}(v)$ into disjoint nonempty subsets $\xi^v_1, \cdots, \xi^v_{m(v)}$ where $m(v)\geq1$ (If $v$ is a source then we
put $m(v) = 0$). Let $\mathcal{P}$ denote the resulting partition of $Q^1$. We form the in-split quiver $Q_r(\mathcal{P})$ from $Q$
using the partition $\mathcal{P}$ as follows:

$$Q_r(\mathcal{P})^0 = \{v_i| v\in Q^0, 1\leq i\leq m(v)\} \cup \{v | m(v)=0\},$$
$$Q_r(\mathcal{P})^1 = \{e_j| e\in Q^1, 1\leq j\leq m(s(e))\} \cup \{e | m(s(e))=0\},$$

and define $r_{Q_r(\mathcal{P})}, s_{Q_r(\mathcal{P})}: Q_r(\mathcal{P})^1\rightarrow Q_r(\mathcal{P})^0$ by

  $s_{Q_r(\mathcal{P})}(e_j)=s(e)_ j$ and $s_{Q_r(\mathcal{P})}(e)=s(e)$,

$r_{Q_r(\mathcal{P})}(e_j)= r(e)_i$ and $r_{Q_r(\mathcal{P})}(e)=r(e)_i$ where $e\in \xi^{r(e)}_i$. Conversely, if $Q$ and $Q'$ are quivers, and there exists a partition $\mathcal{P}$ of $Q^1$ for which $Q_r(\mathcal{P})=Q'$, then $Q$ is called an in-amalgamation of $Q'$.
\end{definition}

\begin{definition}(\cite[Section 3]{BP}, \cite[Definition 1.12]{ALPS})  Let $Q=(Q^0, Q^1, r, s)$ be a quiver. For each $v \in Q^0$ with $s^{-1}(v)\neq \emptyset$, partition the set $s^{-1}(v)$ into disjoint nonempty subsets $\xi^1_v, \cdots, \xi^{m(v)}_v$ where $m(v)\geq1$ (If $v$ is a sink then we
put $m(v) = 0$). Let $\mathcal{P}$ denote the resulting partition of $Q^1$. We form the out-split quiver $Q_s(\mathcal{P})$ from $Q$
using the partition $\mathcal{P}$ as follows:

$$Q_s(\mathcal{P})^0 = \{v^i| v\in Q^0, 1\leq i\leq m(v)\} \cup \{v | m(v)=0\},$$
$$Q_s(\mathcal{P})^1 = \{e^j| e\in Q^1, 1\leq j\leq m(r(e))\} \cup \{e | m(r(e))=0\},$$

and define $r_{Q_s(\mathcal{P})}, s_{Q_s(\mathcal{P})}: Q_s(\mathcal{P})^1\rightarrow Q_s(\mathcal{P})^0$ for each $e\in \xi^i_{s(e)}$ by

  $s_{Q_s(\mathcal{P})}(e^j)=s(e)^i$ and $s_{Q_s(\mathcal{P})}(e)=s(e)^i$,

$r_{Q_s(\mathcal{P})}(e^j)= r(e)^j$ and $r_{Q_s(\mathcal{P})}(e)=r(e)$. Conversely, if $Q$ and $Q'$ are quivers, and there exists a partition $\mathcal{P}$ of $Q^1$ for which $Q_s(\mathcal{P})=Q'$, then $Q$ is called an out-amalgamation of $Q'$.
\end{definition}

The following corollary shows that Proposition 15.(1) of \cite{H1} and Corollary 4.1 of \cite{PS} are corollaries of Theorem \ref{theresult}.

\begin{corollary}\label{cor5} Let $Q$ be a finite quiver with no sinks, $\mathcal{P}$ be a partition of $Q^1$, $Q_r(\mathcal{P})$ the in-split quiver from $Q$ using $\mathcal{P}$ and $Q_s(\mathcal{P})$ the out-split quiver from $Q$ using $\mathcal{P}$. Then $L_k(Q)$, $L_k(Q_r(\mathcal{P}))$ and $L_k(Q_s(\mathcal{P}))$ are $\mathbb{Z}$-graded  Morita equivalent.
\end{corollary}

\begin{proof}
By \cite[Proposition 6.3]{BP} $Q\sim_{ES}Q_r(\mathcal{P})$ and by \cite[Proposition 6.2]{BP} $Q\sim_{ES}Q_s(\mathcal{P})$. Then the result follows by Corollary \ref{cor12}.
\end{proof}

\begin{definition}(\cite[Definition 1.2]{ALPS})
Let $Q=(Q^0, Q^1, r, s)$ be a quiver with at least two vertices, and let $v\in Q^0$ be
a source. We form the source elimination quiver $Q_{\backslash v}$ of $Q$ as follows:
$$Q^0_{\backslash v} = Q^0\backslash\{v\},$$
$$Q^1_{\backslash v} = Q^1\backslash s^{-1}(v),$$
$$s_{Q_{\backslash v}} = s|_{Q^1_{\backslash v}},$$
$$r_{Q_{\backslash v}} = r|_{Q^1_{\backslash v}}.$$
\end{definition}

Hazrat in \cite[Proposition 13]{H1} proved that, if $Q$ is a finite quiver with no sinks and $v\in Q^0$ is a source, then $L_k(Q)$ and $L_k(Q_{\backslash v})$ are graded Morita equivalent. In the following corollary, we generalize this proposition.

\begin{corollary}\label{cor7} Let $Q$ be a quiver with finite number of vertices and $v\in Q^0$ be a source which is not infinite emitter or sink. Then $L_k(Q_{\backslash v})$ and $L_k(Q)$ are $\mathbb{Z}$-graded Morita equivalent.
\end{corollary}

\begin{proof}
Put $Q'^0=Q^0\backslash \{v\}$, then the result follows by Theorem \ref{theresult}.
\end{proof}

Let $\Lambda=k(Q,\rho)$ be a finite dimensional path algebra over a field $k$ of the finite quiver $(Q,\rho)$ with relations. Let $v$ be a source in $Q$ and $\bar{e_v}$ the corresponding idempotent in $\Lambda$. We have $\bar{e_v}\Lambda \bar{e_v}\simeq k$ and $\bar{e_v}\Lambda(1-\bar{e_v})=0$. If $(Q',\rho')$ denotes the quiver with relations we obtain by removing the vertex $v$ and the relations starting at $v$, then $(1-\bar{e_v})\Lambda(1-\bar{e_v})\simeq k(Q',\rho')$. Thus $k(Q,\rho)$ is obtained from $\Lambda'=k(Q',\rho')$ by adding one vertex $v$, together with arrows and relations starting at $v$. Then we have $\Lambda\simeq\left(\begin{array}{cc}
k&0\\
(1-\bar{e_v})\Lambda \bar{e_v}&\Lambda'\\
\end{array}\right)$. $\Lambda$ is called one-point extension of $\Lambda'$ \cite[Page 71]{Au}. \\

By Proposition \ref{pro} and Corollary \ref{cor7} we have the following corollary.

\begin{corollary}\label{cor17} Let $\Lambda$ be a finite dimensional algebra which is one-point extension of $\Lambda'$. Then the algebras $\Lambda/J^2$ and $\Lambda'/J^2$ are singularly equivalent.
\end{corollary}

\begin{definition}(\cite[Definition 1.6]{ALPS})
Let $Q=(Q^0, Q^1, r, s)$ be a quiver and $v\in Q^0$. Let $v^*$ and $f$ be symbols not in $Q^0\cup Q^1$. We form the expansion quiver $Q_v$ from $Q$ at $v$ as follows:

$$Q^0_v= Q^0\cup \{v^*\},$$
$$Q^1_v= Q^1\cup \{f\},$$
$$s_{Q_v}(e) =\left\lbrace
\begin{array}{c l l}
v & if\,\, e=f,\\
v^*  & if \,\, s_Q(e)=v,\\
s_Q(e) & \,\, otherwise,
\end{array}
  \right.$$

$$r_{Q_v}(e) =\left\lbrace
\begin{array}{c l }
v^* & if\,\, e=f,\\
r_Q(e) & \,\, otherwise.
\end{array}
  \right.$$
Conversely, if $Q$ and $Q'$ are quivers, and there exists a vertex $v$ of $Q$ for which $Q_v=Q'$, then $Q$ is called a contraction of $Q'$.
\end{definition}

G. Abrams et al. in \cite[Corollary 3.4]{ALPS} proved that, if $k$ is an infinite field, $Q$ is a row-finite quiver and $v\in Q^0$, then  $L_k(Q)$ and $L_k(Q_{v})$ are Morita equivalent.

\begin{corollary}\label{cor8} Let $Q$ be a quiver with finite number of vertices and $v\in Q^0$ which is not infinite emitter or sink. Then $L_k(Q_v)$ and $L_k(Q)$ are $(1/2)\mathbb{Z}$-graded Morita equivalent.
\end{corollary}

\begin{proof}
Put $Q'^0=Q^0_v\backslash \{v^*\}$, then the result follows by Theorem \ref{theresult}.
\end{proof}

\begin{definition}(\cite[Definitions 1.18 and 1.19]{ALPS}) A quiver transformation is called standard if it is one of these six types: in-splitting,
in-amalgamation, out-splitting, out-amalgamation, expansion, and contraction. Analogously,
a function which transforms a non-negative integer matrix $A$ to a non-negative integer matrix $B$ is called
standard if the corresponding quiver operation from $Q_A$ to $Q_B$ is standard.
If $Q$ and $Q'$ are quivers, a flow equivalence from $Q$ to $Q'$ is a sequence $Q=Q_0\rightarrow Q_1\rightarrow \cdots
\rightarrow Q_n = Q'$ of quivers and standard quiver transformations which starts at $Q$ and ends at $Q'$. We say
that $Q$ and $Q'$ are flow equivalent if there is a flow equivalence from $Q$ to $Q'$. Analogously, a flow
equivalence between matrices $A$ and $B$ is defined to be a flow equivalence between the quivers $Q_A$
and $Q_B$.
\end{definition}

By using Corollaries \ref{cor5}, \ref{cor7} and \ref{cor8} we have the following:

\begin{corollary} Let $Q$ and $Q'$ be quivers with finite number of vertices and no sinks such that $Q$ and $Q'$ are flow equivalent. Then $L_k(Q)$ and $L_k(Q')$ are $(1/2)\mathbb{Z}$-graded Morita equivalent.
\end{corollary}

\subsection{Delays}

The definitions of this section are borrowed from \cite[section 4]{BP} (see also \cite[Definitions 3.2 and 3.5]{ALPS}. \\

\begin{definition}
Let $Q=(Q^0, Q^1, r, s)$ be a quiver. A map $d_s : Q^0\cup Q^1\rightarrow \mathbb{N}\cup \{\infty\}$ such that:

 $(i)$ if $w\in Q^0$ is not a sink then $d_s(w) = sup\{d_s(e) | s(e) = w\}$, and

 $(ii)$ if $d_s(x) = \infty$ for some $x$ then either $x$ is a sink or $x$ emits infinitely many edges,

is called a Drinen source-vector. Note that only vertices are allowed to have an infinite
$d_s$-value; moreover if $d_s(v) = \infty$ and $v$ is not a sink, then there are edges with source $v$
and arbitrarily large $d_s$-value. From this data we construct a new quiver as follows: Let

$$d_s(Q)^0 = \{v^i | v \in Q^0; 0\leq i\leq d_s(v)\},  and$$
$$d_s(Q)^1 = Q^1\cup \{f(v)^i | 1\leq i\leq d_s(v)\},$$

and for $e\in Q^1$ define $r_{d_s(Q)}(e) = r(e)^0$ and $s_{d_s(Q)}(e) = s(e)^{d_s(e)}$. For $f(v)^i$ define $s_{d_s(Q)}(f(v)^i) =
v^{i-1}$ and $r_{d_s(Q)}(f(v)^i) = v^i$. The resulting quiver $d_s(Q)$ is called the out-delayed
quiver of $Q$ for the Drinen source-vector $d_s$.

 In the out-delayed quiver the original vertices correspond to those vertices with superscript
0. Inductively, the edge $e\in Q^1$ is delayed from leaving $s(e)^0$ and arriving at $r(e)^0$ by a path
of length $d_s(e)$. The Drinen source vector $d_s$ is strictly proper if, whenever $v$ has infinite
valency, there is no $v_i$ with infinite valency unless $i = d_s(v) < \infty$. A Drinen source vector
$d_s$ which gives rise to an out-delayed quiver $d_s(Q)$ which may be constructed using
a finite sequence of strictly proper Drinen source vectors is said to be proper.
\end{definition}

\begin{definition}
Let $Q = (Q^0,Q^1, r, s)$ be a quiver. A map $d_r:Q^0\cup
Q^1\rightarrow \mathbb{N}\cup \{\infty\}$ satisfying:

$(i)$ if $w$ is not a source then $d_r(w) = sup \{d_r(e) | r(e) = w\}$, and

$(ii)$ if $d_r(x) = \infty$ then $x$ is either a source or receives infinitely many edges,

is called a Drinen range-vector. We construct a new quiver $d_r(Q)$ called the in-delayed
quiver of $Q$ for the Drinen range-vector $d_r$ as follows:

$$d_r(Q)^0 =\{v_i | v\in Q^0; 0\leq i\leq d_r(v)\}, and$$
$$d_r(Q)^1 = Q^1\cup \{f(v)_i | 1 \leq i\leq d_r(v)\};$$

and for $e\in Q^1$ we define $r_{d_r(Q)}(e) = r(e)_{d_r(e)}$ and $s_{d_r(Q)}(e) = s(e)_0$. For $f(v)_i$ we define
$s_{d_r(Q)}(f(v)_i) = v_i$ and $r_{d_r(Q)}(f(v)_i) = v_{i-1}$.
\end{definition}

\begin{corollary} Let $Q$ be a quiver with finite number of vertices, $d_s : Q^0\cup Q^1\rightarrow \mathbb{N}\cup \{\infty\}$ be a Drinen
source-vector and $d_r : Q^0\cup Q^1\rightarrow \mathbb{N}\cup \{\infty\}$ be a Drinen
range-vector.\\
i) There exists a positive integer $n$ such that $L_k(d_s(Q))$ and
$L_k(Q)$ are $(1/n)\mathbb{Z}$-graded Morita equivalent if and only
if $d_s$ is proper.\\
ii) There exists a positive integer $n$ such that $L_k(d_r(Q))$ and
$L_k(Q)$ are $(1/n)\mathbb{Z}$-graded Morita equivalent.
\end{corollary}

\begin{proof} i) Put $Q'^0=Q^0$, then it is easy to see that if $d_s$ is strictly proper, then the hypothesis of Theorem \ref{theresult} are satisfied. So the result follows by Theorem \ref{theresult}. If $d_s$ is not proper, then there are at least two vertices $f(v)^i, f(v)^j$ with $0\leq i\leq j\leq d_s(v)$ emitting infinitely many edges. In this case there is an ideal generated by $f(v)^i$ in
$L_k(d_s(Q))$ which was not present in $L_k(Q)$. Then $L_k(d_s(Q))$ is not Morita equivalent (and so is not graded Morita equivalent) to $L_k(Q)$.\\
ii) Put $Q'^0=Q^0$, then the result follows by Theorem \ref{theresult}.
\end{proof}

\begin{definition}(\cite[Definition 9.1]{AT})
Given a quiver $Q$, let $M_nQ$ be the quiver formed from $Q$ by taking
each $v\in Q^0$ and attaching a "head" of length $n-1$ of the form
$$\begin{matrix}
\xymatrix{{\bullet}_{v_{n-1}} \ar[r]^{e_{n-1}^v} & {\bullet}_{v_{n-2}} \ar[r]^{e_{n-2}^v}  & {\bullet}_{v_{n-3}} \ar @{.>}[r]&{\bullet}_{v_2} \ar [r]^{e_2^v} &{\bullet}_{v_1} \ar [r]^{e_1^v}  & {\bullet}_{v}}
\end{matrix}$$
to $Q$.
\end{definition}

\begin{corollary} Let $Q''$ be a quiver with finite number of vertices, then $L_k(Q'')$ and $L_k(M_nQ'')$ are $\mathbb{Z}$-graded Morita equivalent.
\end{corollary}

\begin{proof}
Setting $Q=M_nQ''$ and $Q'^0=Q''^0$. Then the result follows by Theorem \ref{theresult}.
\end{proof}

\begin{definition}(\cite[Definition 2.3.10]{LM})
Let $Q$ be a quiver and $n\geq 2$. The n'th higher edge quiver of $Q$ is a quiver $Q^{[n]}$ with the vertex set equal to the collection of all paths of length $n-1$ in $Q$, and the edge set containing exactly one edge from $e_1e_2\cdots e_{n-1}$ to $f_1f_2\cdots f_{n-1}$ whenever $e_2e_3\cdots e_{n-1}=f_1f_2\cdots f_{n-2}$ (or $r(e_1)=s(f_1)$ if $n=2$), and none otherwise.
\end{definition}

\begin{example}
 Let $Q=Q^{[1]}:$ $$\hskip.5cm 
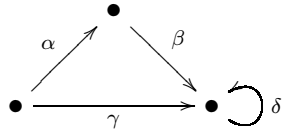
\xymatrix{
&&{\bullet} \ar@{<-}[dl]_{\alpha}\ar[dr]^{\beta}\\
&{\bullet}\ar[rr]_{\gamma}
&&{\bullet} \ar@(dr,ur)_{\delta}}
\hskip.5cm$$ then $Q^{[2]}:$ 
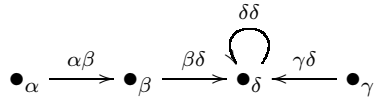
$$\xymatrix{{\bullet}_{\alpha}\ar[r]^{\alpha\beta}&{\bullet}
_{\beta}\ar[r]^{\beta\delta}&{\bullet}_{\delta}\ar@(ur,ul)_{\delta\delta}\ar@{<-}[r]^{\gamma\delta}&{\bullet}_{\gamma}}$$
\end{example}

The next corollary shows that the Corollary 1.4 of \cite{PS} is a corollary of the Theorem \ref{theresult}.

\begin{corollary}\label{cor6} Let $Q$ be a finite quiver with no
sinks. Then for all integers $n\geq 2$, $L_k(Q)$ and $L_k(Q^{[n]})$
are $\mathbb{Z}$-graded Morita equivalent.
\end{corollary}

\begin{proof} Let $Q'$ be a quiver with the vertex set $Q'^0=Q^0\cup Q^1$, and the
edge set containing exactly one edge from $v\in Q^0$ to $e\in Q^1$
whenever $s(e)=v$, exactly one edge from $e\in Q^1$ to $v\in Q^0$
whenever $r(e)=v$, and none otherwise. Then it is easy to see that
$Q\sim_{ES}Q^{[2]}$ via $Q'$. $Q^{[2]}$ and $Q'$ are finite quivers
with no sinks, and so by Corollary \ref{cor12}, $L_k(Q)$
and $L_k(Q^{[2]})$ are $\mathbb{Z}$-graded Morita equivalent. An
easy computation shows that $(Q^{[n-1]})^2=Q^{[n]}$, then by
induction on $n$ we can see that $L_k(Q)$ and $L_k(Q^{[n]})$ are
$\mathbb{Z}$-graded Morita equivalent.
\end{proof}

By Theorem \ref{theresult}, Proposition \ref{pro} and results of
section 2.3 we have the following corollary:

\begin{corollary}\label{cor1} Let $k$ an arbitrary field, $Q$ and $Q'$ be finite quivers without sinks such that one of the following conditions holds:

a) $Q'\sim_{ES}Q$;

b) $Q'\sim_{S}Q$;

c) $Q'=Q_r(\mathcal{P})$ for some partition $\mathcal{P}$ of $Q^1$;

d) $Q'=Q_s(\mathcal{P})$ for some partition $\mathcal{P}$ of $Q^1$;

e) $Q'=Q_{\backslash v}$ for some $v\in Q^0$ which is a source and
not infinite emitter or sink;

f) $Q'=M_nQ$ for some positive integer $n$;

g) $Q'=Q^{[n]}$ for some positive integer $n$;\\

Then the following statements hold:\\

(1) The algebras $kQ/J^2$ and $kQ'/J^2$ are singularly equivalent;

(2) The Leavitt path algebras $L_k(Q)$ and $L_k(Q')$ are derived
equivalent;

(3) The opposite Leavitt path algebras $L_k(Q)^{op}$ and
$L_k(Q')^{op}$ are derived equivalent;

(4) There is a triangle equivalence
$K_{ac}(Inj\textsc{-}kQ/J^2)\xrightarrow[]{\sim}K_{ac}(Inj\textsc{-}kQ'/J^2)$;

(5) There is a triangle equivalence
$K_{ac}(Proj\textsc{-}kQ^{op}/J^2)\xrightarrow[]{\sim}K_{ac}(Proj\textsc{-}kQ'^{op}/J^2)$;

(6) The categories $qgr(kQ)$ and $qgr(kQ')$ are equivalent;

(7) The categories $QGr(kQ)$ and $QGr(kQ')$ are equivalent.
\end{corollary}

\section{Proof of the main theorem}
In this section we give the proof of the Theorem \ref{theresult}.
We need the following lemmas in the proof of the Theorem \ref{theresult}. In the next lemmas we assume that the hypothesis of the Theorem \ref{theresult} are satisfied. The following lemmas are finite version of \cite[Lemma 3.3, Lemma 3.6]{CG}, we give the proof for the convenience of reader.\\

\begin{lemma}\label{firstlemma1} For each $v\in Q^0\backslash Q'^0$ we have $v\geq Q'^0$.
\end{lemma}
\begin{proof} On the contrary, suppose that there exists $v\in Q^0\backslash Q'^0$ such that there is no path in $\text{Path}(Q)$ with source $v$ and range in $Q'^0$. Since $v$ is not a sink, there exists an edge $\mu_1$ with $s(\mu_1)=v$. By assumption $r(\mu_1)\in Q^0\backslash Q'^0$ and $r(\mu_1)\ngeq Q'^0$. Since $Q^0$ is finite, continuing in this way gives a cycle $\mu=\mu_1\mu_2\cdots \mu_n \in\text{Path}(Q)$ with $s(\mu_i)\in Q^0\backslash Q'^0$ for each $i$ which is a contradiction, and so $v\geq Q'^0$.
\end{proof}

Let $v\in Q^0$. We define $\displaystyle B_v:=\{\mu\in \text{Path}(Q)\backslash Q^0 \mid s(\mu)=v, r(\mu)\in Q'^0 \text{ and }  r(\mu_i)\in Q^0\backslash Q'^0  \text{ for each } 1\leq i<|\mu| \}$.

\begin{lemma}\label{firstlemma} (i) If $B_v=\emptyset$ then $v\in Q'^0$.

(ii) If $B_v$ is finite and $B_v\neq \emptyset$, then $0<|s^{-1}_Q(v)|<\infty$.

(iii) Assume that $B_v$ is finite and nonempty. If $\{e, w\mid e\in Q^1, w\in Q^0\}$ is a $Q$-family, then
$v=\sum_{\alpha\in B_v}\alpha\alpha^{\ast}$.

\end{lemma}
\begin{proof}
(i) Let $v\in Q^0\backslash Q'^0$, then by Lemma \ref{firstlemma1} there exists a path $\alpha\in \text{Path}(Q)$ such that $s(\alpha)=v$ and $r(\alpha)\in Q'^0$. Let $i$ be the minimum indices that $r(\alpha_i)\in Q'^0$, then $\alpha_1\alpha_2\cdots\alpha_i\in B_v$ and the result follows.

(ii) Let $\alpha\in s^{-1}_Q(v)$. If $r(\alpha)\in Q'^0$, then $\alpha\in B_v$, and if $r(\alpha)\in Q^0\backslash Q'^0$, then by part (i), there exists $\beta\in B_{r(\alpha)}$ and so $\alpha\beta\in B_v$ is an extension of $\alpha$. Thus $|s^{-1}_Q(v)|<|B_v|<\infty$. Since $B_v\neq \emptyset$, $v$ cannot be a sink and so $0<|s^{-1}_Q(v)|$ and the result follows.

(iii) Put $l(B_v):=max\{|\alpha| : \alpha\in B_v\}$. Since $B_v$ is finite and nonempty, this number is well-defined. We use induction on $l(B_v)$. Assume that $l(B_v)=1$. We show that $B_v=s^{-1}_Q(v)$. Since $l(B_v)=1$, each path in $B_v$ is a single edge and so $B_v\subseteq s^{-1}_Q(v)$. Now assume that $e\in s^{-1}_Q(v)$ and $r(e)\in Q^0\backslash Q'^0$, then by (i), there exists a path $\alpha\in B_{r(e)}$. Thus $e\alpha\in B_v$ is a path of length grater than 2, which is contradiction. Hence $r(e)\in Q'^0$ and so $e\in B_v$. Then $B_v=s^{-1}_Q(v)$ and the CK-2 relation implies that $v=\sum_{e\in s^{-1}_Q(v)}ee^{\ast}=\sum_{\alpha\in B_v}\alpha\alpha^{\ast}$. Now assume that $l(B_v)=n$ and for each $w\in Q^0$ with $1\leq l(B_w)<n$, $w=\sum_{\alpha\in B_w}\alpha\alpha^{\ast}$. $v$ is regular and so
\begin{align*}
v=\sum_{e\in s^{-1}_Q(v)}ee^{\ast}=\sum_{e\in s^{-1}_Q(v)\cap r^{-1}(Q'^0)}ee^{\ast}+\sum_{e\in s^{-1}_Q(v)\cap r^{-1}(Q^0\backslash Q'^0)}ee^{\ast}\\
=\sum_{\alpha\in B_v\cap Q^1}\alpha\alpha^{\ast}+\sum_{e\in s^{-1}_Q(v)\cap r^{-1}(Q^0\backslash Q'^0)}ee^{\ast}.
\end{align*}
Let $e\in s^{-1}_Q(v)\cap r^{-1}(Q^0\backslash Q'^0)$, then (i) implies that $B_{r(e)}\neq \emptyset$. For each $\alpha\in B_{r(e)}$, $e\alpha\in B_v$ and so $B_{r(e)}$ is finite and $l(B_{r(e)})\leq n-1$. Also for each $\gamma\in B_{v}$ with $|\gamma|\geq 2$, there exists some $f\in s^{-1}_Q(v)\cap r^{-1}(Q^0\backslash Q'^0)$ and $\alpha\in B_{r(e)}$ such that $f\alpha=\gamma$. Also $r(\gamma_1)\in s^{-1}_Q(v)\cap r^{-1}(Q^0\backslash Q'^0)$ and $\gamma_2\cdots \gamma_{|\gamma|}\in B_{r(\gamma_1)}$. Then for such $e$ by applying the inductive hypothesis to $r(e)$, we have
\begin{align*}
v=\sum_{\alpha\in B_v\cap Q^1}\alpha\alpha^{\ast}+\sum_{e\in s^{-1}_Q(v)\cap r^{-1}(Q^0\backslash Q'^0)}(\sum_{\beta\in B_{r(e)}}e\beta(e\beta)^{\ast})\\
=\sum_{\alpha\in B_v\cap Q^1}\alpha\alpha^{\ast}+\sum_{\alpha\in B_v, |\alpha|\geq2}\alpha\alpha^{\ast}=\sum_{\alpha\in B_v}\alpha\alpha^{\ast}
\end{align*}
and the result follows.
\end{proof}

\begin{lemma}\label{secondlemma1} Let $\{e, v\mid e\in Q^1, v\in Q^0\}$ be a $Q$-family in $L_k(Q)$, $q_v=v$ for $v\in Q'^0$ and $T_{e_\alpha}=\alpha$ for $e_\alpha\in Q'^1$. Then $\{T_{e_\alpha}, q_v\mid e_\alpha\in Q'^1, v\in Q'^0\}$ is a $Q'$-family in $L_k(Q')$.
\end{lemma}

\begin{proof} The set $\{q_v| v\in Q'^0\}$ is a set of nonzero mutually orthogonal idempotents because the set $\{v| v\in Q'^0\}$ is. For each $e_\alpha\in Q'^1$, $\alpha\in \text{Path}(Q)$ and so $s(T_{e_\alpha})T_{e_\alpha}=T_{e_\alpha}r(T_{e_\alpha})$ and $r(T_{e_\alpha})T_{e_\alpha}^{\ast}=T_{e_\alpha}^{\ast}s(T_{e_\alpha})=T_{e_\alpha}^{\ast}$. Now assume that $e_\alpha, e_\beta\in Q'^1$ and $e_\alpha\neq e_\beta$. Then $\alpha, \beta\in \bigcup_{v\in Q'^0} B_v$. If $\alpha=\beta\gamma$, for some $\gamma\in \text{Path}(Q)\backslash Q^0$, then $|\beta|<|\alpha|$ and $r(\alpha_{|\beta|})=r(\beta)\in Q'^0$, which is contradicting with $\alpha\in B_v$ and so neither one of $\alpha, \beta$ is an extension of the other. Thus by \cite[Lemma 3.1]{T}, $T_{e_\alpha}^{\ast}T_{e_\beta}=\alpha^{\ast}\beta=0$. Also for each $e_\alpha\in Q'^1$, $T_{e_\alpha}^{\ast}T_{e_\alpha}=\alpha^{\ast}\alpha=r(\alpha)=r(T_{e_\alpha})$. Now let a regular vertex $v\in Q'^0$. Then $B_v$ is finite and nonempty since, $0<|s^{-1}_{Q'}(v)|<\infty$. Thus by part (iii) of Lemma \ref{firstlemma}, \begin{equation*}
q_v=v=\sum_{\alpha\in B_v}\alpha\alpha^{\ast}=\sum_{e_\alpha\in s^{-1}_{Q'}(v)}T_{e_\alpha}T_{e_\alpha}^{\ast}.
\end{equation*}
and the result follows.
\end{proof}

A homogeneous element $e$ of a graded ring $A$ is called full homogeneous idempotent if $e^2=e$ and $AeA = A$. The following lemma proved in \cite{H1} by R. Hazrat.

\begin{lemma}\label{secondlemma}(\cite[Example 2.]{H1}) Let $A$ be a graded ring and $e$ be a full homogeneous idempotent of $A$. Then $eAe$ is graded
Morita equivalent to $A$.
\end{lemma}

Now we are ready to proof our main result. In the proof of the Theorem \ref{theresult} we use the general idea of \cite[Theorem 4.2]{BP}. Several authors in \cite[Theorem 3.1]{CG}, \cite[Theorem 3.3]{ALPS} and \cite[Proposition 13.]{H1} use this idea, but details of proofs case by case are different.

\begin{proof}[Proof of Theorem \ref{theresult}] Let $\mathcal{A}=\{\mu|\mu\in\text{Path}(Q)\backslash Q^0, s(\mu), r(\mu)\in Q'^0\hskip.1cm
and\hskip.1cm  r(\mu_i)\in Q^0\backslash Q'^0\\
for\hskip.1cm 1\leq i<l(\mu)\}$ and $n=max\{l(\mu)|\mu\in \mathcal{A}\}$. First we show that the Leavitt path algebras $L_k(Q')$ and $L_k(Q)$ have a $(1/n)\mathbb{Z}$-graded structure. Let $G=(1/n)\mathbb{Z}$ and set $w(v)=0$ for each $v\in Q'^0$, $w(\alpha)=1$ and $w(\alpha^\ast)=-1$ for each $\alpha\in Q'^1$. Then we obtain a natural $(1/n)\mathbb{Z}$-grading on $L_k(Q')$ with the support $\mathbb{Z}$. Set $w(v)=0$ for each $v\in Q^0$. For any $\mu=\mu_1\cdots\mu_n\in \mathcal{A}$, set $w(\mu_i)=1/n$. Now for any $\mu=\mu_1\cdots\mu_{n-1}\in \mathcal{A}
$, we define inductively weight of $\mu_i$ for each $i$. If in the first step we set $w(\mu_{1})=\cdots=w(\mu_{j})=1/n$ (after reordering the $\mu_i$'s if needed) for some $j\leq n-1$, then we set $w(\mu_r)=1/n $ for each $j+1\leq r\leq n-2$ and $w(\mu_{n-1})=2/n$. Otherwise we set $w(\mu_{i})=1/n$ for each $1\leq i\leq n-2$ and $w(\mu_{n-1})=2/n$. We obtain inductively the weight of any edge of $Q$ which lies on some path $\mu=\mu_1\cdots\mu_{j}\in \mathcal{A}$, for some $1\leq j\leq n$. For any edge $\alpha$ of $Q$ which dose not lie on the paths $\mu=\mu_1\cdots\mu_{j}\in \mathcal{A}$, for each $1\leq j\leq n$, we set $w(\alpha)=1$. Also for any $\alpha\in Q^1$, we set $w(\alpha^\ast)=-w(\alpha)$. Then we obtain a natural $(1/n)\mathbb{Z}$-grading on $L_k(Q)$. Note that by this grading for any $\mu\in \mathcal{A}$, $deg(\mu)=1$.

Let $\{e, v\mid e\in Q^1, v\in Q^0\}$ be a $Q$-family in $L_k(Q)$,  $q_v=v$ for $v\in Q'^0$ and $T_{e_\alpha}=\alpha$ for $e_\alpha\in Q'^1$. By Lemma \ref{secondlemma1}, $\{T_{e_\alpha}, q_v\mid e_\alpha\in Q'^1, v\in Q'^0\}$ is a $Q'$-family in $L_k(Q')$. Put $\{e_\alpha, v\}$ be the canonical $Q'$-family. The Universal Homomorphism Property of $L_k(Q')$ implies that, there exists a $k$-algebra homomorphism $$\pi:L_k(Q')\rightarrow L_k(\{T_{e_\alpha}, q_v\})$$ such that $\pi(e_\alpha)=T_{e_\alpha}$ and $\pi(v)=q_v$, for each $e_\alpha\in Q'^1$ and $v\in Q'^0$. For each $e_\alpha\in Q'^1$ we set $w(T_{e_\alpha})=1$, $w(T_{e_\alpha}^\ast)=-1$ and for each $v\in Q'^0$, we set $w(q_v)=0$. Then $L_k(\{T_{e_\alpha}, q_v\})$ is a $\mathbb{Z}$-graded (also $(1/n)\mathbb{Z}$-graded) algebra. Note that $L_k(Q')$ is a $(1/n)\mathbb{Z}$-graded algebra with the support $\mathbb{Z}$. It is not difficult to see that $\pi$ is a $\mathbb{Z}$-graded ring homomorphism with $\pi(v)=q_v\neq0$ for each $v\in Q'^0$. The Graded Uniqueness Theorem \cite[Theorem 4.8]{T}, implies that $\pi$ is a $\mathbb{Z}$-graded isomorphism. Suppose that $Q'^0=\{v_1, v_2\cdots,v_m\}$ and put $e=\sum_{i=1}^{m}v_i$, then $e$ is an idempotent in $L_k(\{T_{e_\alpha}, q_v\})$. We claim that $eL_k(Q)e=L_k(\{T_{e_\alpha}, q_v\})$. For each $v\in Q'^0$, obviously $q_v=v\in eL_k(Q)e$ and for each $e_\alpha\in Q'^1$, $T_{e_\alpha}=eT_{e_\alpha}e\in eL_k(Q)e$. Then $eL_k(Q)e\supseteq L_k(\{T_{e_\alpha}, q_v\})$. Now let $\lambda\mu^{\ast}\in L_k(Q)$. Suppose that $e\lambda\mu^{\ast}e\neq 0$, then $e\lambda\mu^{\ast}e=\lambda\mu^{\ast}$, $s(\lambda), s(\mu)\in Q'^0$ and $r(\lambda)=r(\mu)$. We consider the following two cases to show that $e\lambda\mu^{\ast}e\in L_k(\{T_{e_\alpha}, q_v\})$.

$\circ$ Case 1: $r(\lambda)\in Q'^0$. We first show that $\lambda$ is a product of paths in $\bigcup_{i=1}^mB_{v_i}$. Assume that $\lambda=\lambda_1\lambda_2\cdots\lambda_t$. If $t=1$, by definition of $B_v$ the assertion holds. Now suppose that the assertion holds for paths of length less than $t$. If there exists $1\leq i<t$ such that $r(\lambda_i)\in Q'^0$, then $\lambda=(\lambda_1\cdots \lambda_i)(\lambda_{i+1}\cdots\lambda_t)$ and the result follows by inductive hypothesis. If there is no such $i$, then $\lambda\in B_{s(\lambda)}$ by definition. Thus $\lambda\in L_k(\{T_{e_\alpha}, q_v\})$. Since $r(\mu)=r(\lambda)\in Q'^0$, the same argument shows that $\mu\in L_k(\{T_{e_\alpha}, q_v\})$ and so $e\lambda\mu^{\ast}e=\lambda\mu^{\ast}\in L_k(\{T_{e_\alpha}, q_v\})$.

$\circ$ Case 2: $r(\lambda)\in Q^0\backslash Q'^0$. Since
$r(\lambda)\in Q^0\backslash Q'^0$ and $Q^0_{\text{sing}}\subseteq
Q'^0$ by definition of $B_{r(\lambda)}$,  $B_{r(\lambda)}$ is a
finite set and by part (i) of Lemma \ref{firstlemma},
$B_{r(\lambda)}\neq \emptyset$. Let $j=max\{i| s(\lambda_i)\in
Q'^0\}$ and $\alpha=\lambda_1\cdots \lambda_{j-1}$,
$\beta=\lambda_j\cdots \lambda_{|\lambda|}$. So we have
$\lambda=\alpha\beta$ and $s(\alpha), r(\alpha), s(\beta)\in Q'^0$
and for each $1\leq i \leq  |\beta|$, $r(\beta_i)\in Q^0\backslash
Q'^0$. Similarly we can decompose $\mu$ as a product
$\gamma\delta$ for some paths $\gamma, \delta\in \text{Path}(Q)$
with the same properties as $\alpha$ and $\beta$, respectively. Case
1 shows that $\alpha, \gamma\in L_k(\{T_{e_\alpha}, q_v\})$. Since
$B_{r(\lambda)}$ is finite and nonempty, and
$r(\lambda)=r(\beta)=r(\delta)$, by part (iii) of Lemma
\ref{firstlemma} we have $\beta\delta^{\ast}=\beta
r(\lambda)\delta^{\ast}=\beta (\sum_{\eta\in
B_{r(\lambda)}}\eta\eta^{\ast})\delta^{\ast} =\sum_{\eta\in
B_{r(\lambda)}}\beta\eta(\delta\eta)^{\ast}=\sum_{\eta\in
B_{r(\lambda)}}T_{e_{\beta\eta}}T_{e_{\delta\eta}}^{\ast}$. In the
last equation we use the fact that $\beta\eta, \delta\eta\in
\bigcup_{i=1}^mB_{v_i}$. Thus $\beta\delta^{\ast}\in
L_k(\{T_{e_\alpha}, q_v\})$ and hence $\lambda\mu^{\ast}$ is a
finite sum of elements of $L_k(\{T_{e_\alpha}, q_v\})$. Then
$\lambda\mu^{\ast}\in L_k(\{T_{e_\alpha}, q_v\})$. Now let
$a\in eL_k(Q)e$. By \cite[Corollary 3.2]{T}, each $x\in L_k(Q)$ has
the form $x=\sum_{i=1}^{t}\lambda_i\alpha_i\beta_i^{\ast}$, such
that for each $1\leq i\leq t$, $\alpha_i, \beta_i\in
\text{Path}(Q)$, $r(\alpha_i)=r(\beta_i)$ and $\lambda_i\in k$. Thus
$a=\sum_{i=1}^{t}\lambda_i(e\alpha_i\beta_i^{\ast}e)$. Cases 1
and 2 show that each $e\alpha_i\beta_i^{\ast}e\in
L_k(\{T_{e_\alpha}, q_v\})$ and so $a\in L_k(\{T_{e_\alpha},
q_v\})$. Then $eL_k(Q)e=L_k(\{T_{e_\alpha},
q_v\})$ and our claim follows.

Now we show that $Q^0$ is the hereditary saturated closure of
$B=\{v_1, v_2\cdots, v_m\}$. We have $Q'^0\subseteq \overline{B}$.
Let $v\in Q^0\backslash Q'^0$, then by part $(i)$ of Lemma
\ref{firstlemma}, $B_v\neq\emptyset$. The same argument as in the
Case 2 shows that $B_v$ is finite. Then by Lemma \ref{firstlemma},
$0<|s^{-1}_Q(v)|<\infty$. Let $l(B_v):=max\{|\alpha| : \alpha\in
B_v\}$ as in the proof of Lemma \ref{firstlemma}. By induction on
$l(B_v)$ we show that $v\in \overline{B}$. If $l(B_v)=1$ then the
same argument as in the proof of Lemma \ref{firstlemma} implies that
for each $\alpha\in s^{-1}(v)$, $r(\alpha)\in Q'^0$. Since
$0<|s^{-1}_Q(v)|<\infty$, $v\in \overline{B}$. Now suppose that
for each  $w\in Q^0\backslash Q'^0$ if $1\leq l(B_w)\leq n$ then
$w\in \overline{B}$, and assume that $l(B_v)=n+1>1$. For each $\alpha\in
s^{-1}(v)$, we have either $r(\alpha)\in Q^0$ or $r(\alpha)\in Q^0\backslash
Q'^0$. Suppose that $r(\alpha)\in Q^0\backslash Q'^0$, then by Lemma
\ref{firstlemma}, $B_{r(\alpha)}\neq \emptyset$ and so $l(B_{r(\alpha)})\geq
1$. Thus $1\leq l(B_{r(\alpha)})\leq n$, since each path $\beta\in
B_{r(\alpha)}$ gives a path $\alpha\beta\in B_v$ with length $|\beta|+1$ and
hence $r(\alpha)\in \overline{B}$ by inductive hypothesis. Then for any
edge $\alpha\in s^{-1}(v)$, $r(\alpha)\in \overline{B}$, and so $v\in
\overline{B}$, since $v$ is regular. Thus $Q^0\backslash Q'^0
\subseteq \overline{B}$ and so $Q^0$ is the hereditary saturated
closure of $B=\{v_1, v_2\cdots,v_m\}$. By \cite[Lemma 2.1]{APS},
$\sum_{i=1}^{m}L_k(Q)v_iL_k(Q)=L_k(Q)$ and so $e=v_1+\cdots+v_m$
is a full homogeneous idempotent of $L_k(Q)$. Then by Lemma
\ref{secondlemma}, $L_k(Q)$ and $eL_k(Q)e$ are
$(1/n)\mathbb{Z}$-graded Morita equivalent. Thus $L_k(Q)$ and $L_k(Q')$
are $(1/n)\mathbb{Z}$-graded Morita equivalent and the result
follows.
\end{proof}


\section*{acknowledgements}
Special thanks are due to the referee who read this
paper carefully and made useful comments and suggestions that improved the presentation of the paper.
The author would like to thank Professor R. Hazrat for valuable discussions concerning Theorem \ref{theresult} and Professor G. Bergman for valuable discussions and helps concerning Proposition \ref{l}.

\end{document}